\documentclass[12pt]{article}
\usepackage{geometry} % see geometry.pdf on how to lay out the page. There's lots.
\usepackage{amssymb}
\usepackage{amsmath}
\usepackage{mathrsfs}
\usepackage{amsthm}
\theoremstyle{proclaim}
\newtheorem{theorem}{Theorem}[section]
\newtheorem{lemma}[theorem]{Lemma}

\newtheorem{proposition}[theorem]{Proposition}

\theoremstyle{remark}

\newcommand{\eps}{\epsilon}

\newcommand{\cball}[1]{B_{#1}}
\newcommand{\dual}{\sp{\ast}}

\DeclareMathOperator{\diam}{{diam}}

\numberwithin{theorem}{section}
\numberwithin{equation}{section}

\begin{document}
\title{Szlenk and $w^\ast$-dentability indices of the Banach spaces $C([0,\,\alpha])$}
\author{Philip A.H. Brooker\footnote{The author gratefully acknowledges the financial support of a Lift-Off Fellowship from the Australian Mathematical Society. This research was undertaken whilst the author was a Visiting Fellow at the Mathematical Sciences Institute of the Australian National University.} \footnote{E-mail address: philip.a.h.brooker@gmail.com}}

\date{February 24, 2012} 
\maketitle

\begin{abstract}Let $\alpha$ be an infinite ordinal and $\gamma$ the unique ordinal satisfying $\omega^{\omega^\gamma}\leq \alpha < \omega^{\omega^{\gamma+1}}$. We show that the Banach space $C([0,\,\alpha])$ of all continuous scalar-valued functions on the compact ordinal interval $[0,\,\alpha]$ has Szlenk index equal to $\omega^{\gamma+1}$ and $w^\ast$-dentability index equal to $\omega^{1+\gamma+1}$.\end{abstract}

\section{Introduction}\label{secone}
The Szlenk index is an ordinal-valued isomorphic invariant of a Banach space that was introduced in \cite{Szlenk1968}. There it is used to show that the class of separable, reflexive Banach spaces contains no universal element, thereby solving a problem posed by Banach and Mazur in the Scottish Book. Since then the Szlenk index has found a variety of uses in the study of Banach space geometry, a survey of which can be found in \cite{Lancien2006}. One of the main applications of the Szlenk index is in the study of $C(K)$ spaces and their operators, as witnessed in particular by the work of Alspach \cite{Alspach1978}, Alspach and Benyamini \cite{Alspach1979}, Benyamini \cite{Benyamini1978}, Bourgain \cite{Bourgain1979} and Gasparis \cite{Gasparis2005}; we refer to the survey article \cite{Rosenthal2003} for a detailed discussion of this topic. The purpose of the current paper is to enlarge the class of $C(K)$ spaces for which the Szlenk index of $C(K)$ is known. We shall also discuss the related $w^\ast$-dentability index for the same class of $C(K)$ spaces.

It is a classical result of Mazurkiewicz and Sierpinski \cite{MS} that every countable, compact Hausdorff space is homeomorphic to an ordinal interval $[0,\,\alpha]$ equipped with its order topology, for some $\alpha<\omega_1$. The linear isomorphic classification of $C(K)$ spaces with $K$ countable is due to Bessaga and Pe{\l}czy{\'n}ski \cite{Bessaga1960}, who showed that for ordinals $\omega \leq \alpha <\beta <\omega_1$, $C([0,\,\alpha])$ is isomorphic to $C([0,\,\beta])$ if and only if $\beta<\alpha^\omega$. In particular, it follows that each $C(K)$ space with $K$ countable is isomorphic to the space $C([0,\,\omega^{\omega^\gamma}])$ for a unique countable ordinal $\gamma$. The computation of the Szlenk indices of the Banach spaces $C(K)$ with $K$ countable is due to Samuel \cite{Samuel1984}; drawing upon deep results of Alspach and Benyamini \cite{Alspach1979}, Samuel showed that the Szlenk index of $C([0,\,\omega^{\omega^\gamma}])$ is $\omega^{\gamma+1}$ for each countable ordinal $\gamma$.

The first extension of Samuel's result was achieved by Lancien \cite{Lancien1996}, who used Samuel's result and a separable-reduction argument to show that if $K$ is a (scattered) compact Hausdorff space of countable height, then the Szlenk index of $C(K)$ is $\omega^{\gamma+1}$, where $\gamma$ is the unique ordinal such that the height of $K$ belongs to the ordinal interval $[\omega^\gamma, \, \omega^{\gamma+1})$. H{\'a}jek and Lancien later gave in \cite{H'ajek2007} a `direct' proof of Samuel's result, the existence of which was conjectured by Rosenthal in \cite{Rosenthal2003}; in particular, they computed the Szlenk indices of the Banach spaces $C([0,\,\alpha])$ for ordinals $\alpha<\omega_1\omega$ without appeal to Alspach and Benyamini's results from  \cite{Alspach1979}. Their result says that for $\omega\leq \alpha<\omega_1\omega$ the Szlenk index of $C([0,\,\alpha])$ is $\omega^{\gamma+1}$, where $\gamma$ is the unique ordinal satisfying $\omega^{\omega^\gamma}\leq \alpha<\omega^{\omega^{\gamma+1}}$. As the Szlenk index of $c_0(X)$ coincides with the Szlenk index of $X$ for every infinite dimensional Banach space $X$, it follows easily that the statement of H{\'a}jek and Lancien's result holds in fact for all ordinals $\alpha<\omega_1\omega^\omega$ (see \cite[p.2232]{Brookera} for details.)

The main purpose of the current paper is to determine the Szlenk index of $C([0,\,\alpha])$ for $\alpha$ an arbitrary ordinal. In particular, we shall extend the previous results of Samuel and H{\'a}jek-Lancien, showing that for $\alpha\geq \omega$ the Szlenk index of $C([0,\,\alpha])$ is $\omega^{\gamma+1}$, where $\gamma$ is the unique ordinal satisfying $\omega^{\omega^\gamma}\leq \alpha<\omega^{\omega^{\gamma+1}}$ (Theorem~\ref{stpip}). The computation by H{\'a}jek and Lancien of Szlenk indices of the spaces $C([0,\,\alpha])$, $\alpha<\omega_1$, makes use of the aforementioned isomorphic classification of the spaces $C([0,\,\alpha])$, $\alpha<\omega_1$, by Bessaga and Pe{\l}czy{\'n}ski. That the statement of the Bessaga-Pe{\l}czy{\'n}ski classification theorem does not hold in general for the spaces $C([0,\,\alpha])$ when $\alpha\geq\omega_1$ is the reason that the argument of H{\'a}jek and Lancien does not yield the Szlenk indices of the spaces $C([0,\,\alpha])$ for arbitrary $\alpha$. In the current paper we avoid an appeal to the Bessaga-Pe{\l}czy{\'n}ski theorem, working instead through decompositions of the spaces $C([0,\,\alpha])$ into $c_0$-direct sums of smaller spaces of continuous functions on compact ordinals (cf. Lemma~\ref{bp}) and isomorphisms of such $c_0$-direct sums (cf. Lemma~\ref{pb}).

In Section~\ref{secthree} we shall outline how the techniques developed in \cite{H'ajek2009} can be combined with the arguments used in the proof of Theorem~\ref{stpip} of the current paper to show that for $\alpha$ and $\gamma$ as in the preceding paragraph, the $w^\ast$-dentability index of $C([0,\,\alpha])$ is $\omega^{1+\gamma+1}$ (Theorem~\ref{dentthm}).

We now detail most of the necessary terminology and background results for the current paper. As usual, $\omega$ denotes the first infinite ordinal and $\omega_1$ denotes the first uncountable ordinal. For $K$ a compact Hausdorff space, $C(K)$ is the Banach space of continuous scalar-valued functions on $K$, equipped with the supremum norm. For $\alpha$ an ordinal, the ordinal interval $[0,\,\alpha]$ is a compact Hausdorff space when equipped with its order topology. We denote by $C_0([0,\,\alpha])$ the closed subspace of $C([0,\,\alpha])$ consisting of all elements of $C([0,\,\alpha])$ that vanish at $\alpha$. It is well-known and easy to show that $C_0([0,\,\alpha])$ is isomorphic to $C([0,\,\alpha])$ whenever $\alpha\geq\omega$. For ordinals $\xi \leq \alpha$ and $f\in C_0([0,\,\xi])$, we define $f_{\xi,\,\alpha}\in C_0([0,\,\alpha])$ by setting
\[
f_{\xi,\,\alpha}(\zeta)= \begin{cases}
f(\zeta)& \text{if $\zeta\leq \xi$},\\
0& \text{if $\zeta >\xi$,}
\end{cases}\quad 0\leq \zeta \leq\alpha\, .
\]
It is clear that the operator $J_{\xi,\,\alpha}: C_0([0,\,\xi])\longrightarrow C_0([0,\,\alpha]): f\mapsto f_{\xi,\,\alpha}$ is an isometric linear embedding of $C_0([0,\,\xi])$ into $C_0([0,\,\alpha])$.

For a Banach space $X$ we write $\cball{X}$ for the set $\{ x\in X\mid \Vert x\Vert\leq 1\}$. If $Y$ is a Banach space that is isomorphic to $X$, we write $X\approx Y$. If $S$ is a nonempty set, $c_0(S, \, X)$ is defined to be the linear space
\[
\{ f: S\longrightarrow X \mid \{ s\in S \mid \Vert f(s)\Vert >\eps\} \mbox{ is finite for every }\eps>0\}
\]
which we equip with the complete norm $\Vert f\Vert := \sup\{ \Vert f(s)\Vert\mid s \in S\}$. For a nonempty subset $R\subseteq S$, we denote by $U_R$ the canonical isometric linear embedding of $c_0(R, \, X)$ into $c_0(S, \, X)$. The dual space $c_0(S,\, X)\dual$ is naturally identified via isometric linear isomorphism with the Banach space $\ell_1(S,\,X\dual)$.

The Szlenk index is defined as follows. Let $X$ be an Asplund space and $B \subseteq X\dual$. Define
\begin{equation*}
s_\eps(B) = \{x \in B \mid \diam (B \cap V)> \eps \mbox{ for every } w\dual \mbox{-open }V\ni x\}\,.
\end{equation*} We iterate $s_\eps$ transfinitely as follows: $s_\eps^0(B) = B$, $ s_\eps^{\beta+1}(B)= s_\eps(s_\eps^\beta(B))$ for each ordinal $\beta$ and $ s_\eps^\beta(B) = \bigcap_{\sigma< \beta} s_\eps^\sigma(B)$ whenever $\beta $ is a limit ordinal. 

The \emph{$\eps$-Szlenk index of $X$}, denoted ${Sz}(X, \, \eps)$, is the first ordinal $\beta$ such that $s_\eps^\beta(\cball{X\dual}) = \emptyset$. The \emph{Szlenk index of $X$} is the ordinal $Sz(X):=\sup_{\eps >0}Sz(X, \, \eps)$. Note that $Sz(X,\,\eps)$ exists for every Asplund space $X$ and $\eps > 0$ by following well-known characterisation of Asplund spaces: a Banach space is Asplund if and only if every bounded subset of its dual admits $w\dual$-open slices of arbitrarily small norm diameter \cite[Theorem~5.2]{Deville1993}. The ordinal index $Sz(X)$ is thus defined for every Asplund space $X$, and the definition cannot be extended beyond the class of Asplund spaces. It is worth noting that the definition of the Szlenk index used in the current paper (and many others) differs from that introduced by Szlenk in \cite{Szlenk1968}, however the two definitions give the same index on separable Banach spaces containing no copy of $\ell_1$.

The following proposition collects some basic facts regarding the Szlenk index.
\begin{proposition}\label{collection}
Let $X$ and $Y$ be Asplund spaces.
\begin{itemize}
\item[(i)] If $X$ is isomorphic to a subspace of $Y$, then $ Sz(X)\leq Sz(Y)$. In particular, the Szlenk index is an isomorphic invariant of an Asplund space.

\item[(ii)] If $\gamma$ is an ordinal and $\eps >0$ is such that $Sz(X, \, \eps)>\omega^\gamma$, then $Sz(X)\geq \omega^{\gamma+1}$. It follows that $Sz(X)=\omega^\alpha$ for some ordinal $\alpha$.

\item[(iii)] $Sz(X)=1$ if and only $\dim (X)<\infty$.
\end{itemize}
\end{proposition}

Details of the proofs of assertions (i) and (ii) of Proposition~\ref{collection} can be found in \cite[\S2.4]{H'ajek2008}. Verification of (iii) is elementary.

The characterisation of those compact Hausdorff spaces $K$ for which $C(K)$ is an Asplund space is due to Namioka and Phelps; they showed in \cite{Namioka1975} that a Banach space $C(K)$ is Asplund if and only if $K$ is scattered. As every ordinal interval $[0,\,\alpha]$ is scattered and compact when equipped with its order topology, the spaces $C([0,\,\alpha])$ are Asplund spaces and their Szlenk index is defined. Information regarding topological properties of ordinals can be found in, e.g., \cite[\S8.6]{Semadeni1971}.

Important to our analysis of the spaces $C([0,\,\alpha])$ and their duals is the fact that for a scattered, compact Hausdorff space $K$, the dual space $C(K)\dual$ is naturally identified with $\ell_1(K)$; this is due to Rudin \cite[Theorem~6]{Rudin1957}. The dual of $C_0([0,\,\alpha])$ is naturally identified with $\ell_1([0,\,\alpha))$. 

\section{The Szlenk index of $C([0,\,\alpha])$}\label{sectwo}
We begin our computations of Szlenk indices by gathering some preliminary results that we shall need. The first such result is the following proposition that provides a way to obtain an upper estimate of the Szlenk index of a Banach space.
\begin{proposition}[\cite{H'ajek2007}]\label{upest}
Let $X$ be a Banach space and $\eta$ an ordinal. Assume that
\[
\forall\eps>0\quad \exists\delta(\eps)\in(0,\,1) \quad s_\eps^\eta (\cball{X\dual}) \subseteq (1-\delta(\eps))\cball{X\dual}\, .
\]
Then
\[
Sz(X)\leq \eta\omega\, .
\]
\end{proposition}

\begin{lemma}\label{bp}
Let $\xi$ and $\zeta$ be ordinals satisfying $0<\zeta\leq\xi$ and $\omega\leq\xi$. Then $C_0([0,\,\xi\zeta])\approx C_0([0,\,\zeta])\oplus c_0(\zeta, \, C_0([0,\,\xi]))$.
\end{lemma}
Lemma~\ref{bp} is essentially noted by Bessaga and Pe{\l}czy{\'n}ski in their proof of \cite[2.4]{Bessaga1960}; for the sake of completeness, we give here the details of their sketch proof.
\begin{proof}
We may write $C_0([0,\,\xi\zeta])=Y\oplus Z$, where $Y$ consists of all elements of $C_0([0,\,\xi\zeta])$ that are constant on the intervals $(\xi \sigma,\,\xi(\sigma+1)]$, $0\leq \sigma<\zeta$, and $Z$ consists of all elements of $C_0([0,\,\xi\zeta])$ vanishing at the points $\xi\sigma$, $1\leq \sigma\leq \zeta$. The lemma then follows from the routine observation that $Y$ and $Z$ are isometrically isomorphic to $C_0([0,\,\zeta])$ and $ c_0(\zeta, \, C_0([0,\,\xi]))$ respectively.
\end{proof}

We have the following consequence of Lemma~\ref{bp}.
\begin{lemma}\label{pb}
Let $\gamma$ be an ordinal and $1<n<\omega$. Then \[ C_0([0,\,\omega^{\omega^\gamma n}])\approx c_0(\omega^{\omega^\gamma}, \, C_0([0,\,\omega^{\omega^\gamma}]))\, .\]
\end{lemma}
\begin{proof}
We proceed via induction on $n$. For the case $n=2$, note that an application of Lemma~\ref{bp} with $\xi=\zeta = \omega^{\omega^\gamma}$ yields
\[
C_0([0,\,\omega^{\omega^\gamma 2}])\approx C_0([0,\,\omega^{\omega^\gamma}])\oplus c_0(\omega^{\omega^\gamma}, \, C_0([0,\,\omega^{\omega^\gamma}])) \approx c_0(\omega^{\omega^\gamma}, \, C_0([0,\,\omega^{\omega^\gamma}]))\, ,
\]
as desired. Similarly, if $C_0([0,\,\omega^{\omega^\gamma n}])\approx c_0(\omega^{\omega^\gamma}, \, C_0([0,\,\omega^{\omega^\gamma}]))$, then applying Lemma~\ref{bp} with $\zeta= \omega^{\omega^\gamma}$ and $\xi = \omega^{\omega^\gamma n}$ yields
\begin{align*}
C_0([0,\,\omega^{\omega^\gamma (n+1)}])
&\approx C_0([0,\,\omega^{\omega^\gamma}]) \oplus c_0(\omega^{\omega^\gamma}, \, C_0([0,\,\omega^{\omega^\gamma n}]))\\
&\approx C_0([0,\,\omega^{\omega^\gamma}]) \oplus c_0(\omega^{\omega^\gamma},\, c_0(\omega^{\omega^\gamma}, \, C_0([0,\,\omega^{\omega^\gamma}])))\\
&\approx C_0([0,\,\omega^{\omega^\gamma}]) \oplus c_0(\omega^{\omega^\gamma}, \, C_0([0,\,\omega^{\omega^\gamma}]))\\
&\approx c_0(\omega^{\omega^\gamma}, \, C_0([0,\,\omega^{\omega^\gamma}]))\,,
\end{align*}
which completes the proof.
\end{proof}

The last of the preliminary results that we shall require is the following generalisation of \cite[Lemma~3.3]{H'ajek2007}.

\begin{lemma}\label{HaLag}
Let $\alpha$, $\beta$ and $\xi$ be ordinals such that $\xi<\alpha$, let $S$ be a set, $\emptyset \subsetneq R\subseteq S$ and $\eps >0$. If $(z_s)_{s\in S}\in s_{3\eps}^\beta(\cball{c_0(S,\,C_0([0,\,\alpha]))\dual})$ and $\sum_{r\in R}\Vert  J_{\xi,\,\alpha}\dual z_r\Vert>1-\eps$, then $(J_{\xi,\, \alpha}\dual z_r)_{r\in R}\in s_\eps^\beta (B_{c_0(R,\,C_0([0,\,\xi]))\dual})$.
\end{lemma}

\begin{proof}
We proceed via transfinite induction on $\beta$. The assertion of the lemma is clearly true for $\beta=0$. Suppose that $\sigma$ is an ordinal such that the assertion of the lemma holds for all $\beta\leq \sigma$; we will show that the lemma holds also for $\beta=\sigma+1$. Let $(z_s)_{s\in S} \in \cball{c_0(S,\,C_0([0,\,\alpha]))\dual}$ be such that $\sum_{r\in R}\Vert J_{\xi,\,\alpha}\dual z_r \Vert>1-\eps$ and $(J_{\xi,\, \alpha}\dual z_r)_{r\in R}\notin s_\eps^{\sigma+1} (B_{c_0(R,\,C_0([0,\,\xi]))\dual})$. Since we intend to show that $(z_s)_{s\in S} \notin s_{3\eps}^{\sigma+1}(\cball{c_0(S,\,C_0([0,\,\alpha]))\dual})$, we may assume that $(z_s)_{s\in S} \in s_{3\eps}^{\sigma}(\cball{c_0(S,\,C_0([0,\,\alpha]))\dual})$, hence $(J_{\xi,\, \alpha}\dual z_r)_{r\in R}\in s_\eps^\sigma (B_{c_0(R,\,C_0([0,\,\xi]))\dual})$ by the induction hypothesis. So there is a $w\dual$-open subset $V$ of $c_0(R, \, C_0([0,\,\xi]))\dual$ containing $(J_{\xi,\, \alpha}\dual z_r)_{r\in R}$ and such that $\diam (V\cap s_\eps^\sigma(\cball{c_0(R,\, C_0([0,\,\xi]))\dual}))\leq \eps$. Since $\sum_{r\in R}\Vert J_{\xi,\,\alpha}\dual z_r \Vert>1-\eps$, we may assume that
\begin{equation}\label{doop}
V\cap (1-\eps)\cball{c_0(R,\,C_0([0,\,\xi]))\dual} =\emptyset\, .
\end{equation}
Define
\[
J: c_0(R,\, C_0([0,\,\xi]))\longrightarrow c_0(R,\,C_0([0,\,\alpha])): (x_r)_{r\in R}\mapsto (J_{\xi, \, \alpha}x_r)_{r\in R}\,,
\]
so that $U_RJ$ is an isometric linear embedding of $c_0(R, \, C_0([0,\,\xi]))$ into $c_0(S, \, C_0([0,\,\alpha]))$. Let $W=(J\dual U_R\dual)^{-1}(V)$, so that $W$ is a $w\dual$-open set containing $(z_s)_{s\in S}$, and let $(u_s)_{s\in S}\in W\cap s_{3\eps}^\sigma(\cball{c_0(S,\, C_0([0,\,\alpha]))\dual})$. Then $\sum_{r\in R}\Vert J_{\xi, \alpha}\dual u_r \Vert >1-\eps$ by (\ref{doop}) and $(u_s)_{s\in S}\in s_{3\eps}^\sigma(\cball{c_0(S,\, C_0([0,\,\alpha]))\dual})$ by assumption, hence $J\dual U_R\dual (u_s)_{s\in S}\in s_{\eps}^\sigma (\cball{c_0(R, \, C_0([0,\,\xi]))\dual})$ by the induction hypothesis. Suppose $(u_s)_{s\in S}, \, (v_s)_{s\in S}\in W\cap s_{3\eps}^\sigma (\cball{c_0(S, \, C_0([0,\,\alpha]))\dual})$. Then \[ \Vert J\dual U_R\dual (u_s)_{s\in S} - J\dual U_R\dual (v_s)_{s\in S} \Vert \leq \diam (V\cap s_\eps^\sigma(\cball{c_0(R,\, C_0([0,\,\xi]))\dual}))\leq \eps \, .\] Moreover, since $\Vert J\dual U_R\dual(u_s)_{s\in S}\Vert >1-\eps$, we have
\[
\sum_{s\in S\setminus R}\Vert u_s\Vert +\sum_{r\in R}\Vert u_r|_{[\xi, \, \alpha)}\Vert <\eps\, ,
\]
and similarly,
\[
\sum_{s\in S\setminus R}\Vert v_s\Vert +\sum_{r\in R}\Vert v_r|_{[\xi, \, \alpha)}\Vert <\eps\, .
\]
It follows that
\begin{align*}
\Vert (u_s)_{s\in S}- (v_s)_{s\in S}\Vert 
&\leq \Vert J\dual U_R\dual (u_s)_{s\in S} - J\dual U_R\dual (v_s)_{s\in S} \Vert + \sum_{s\in S\setminus R}\Vert u_s-v_s\Vert +\sum_{r\in R}\Vert (u_r-v_r)|_{[\xi, \, \alpha)} \Vert\\
&\leq \eps +\eps +\eps = 3\eps\, .
\end{align*}
In particular, $\diam (W\cap s_{3\eps}^\sigma(\cball{c_0(S, \, C_0([0,\,\alpha]))\dual}))\leq 3\eps$, hence $(z_s)_{s\in S}\notin s_{3\eps}^{\sigma+1}(\cball{c_0(S, \, C_0([0,\,\alpha]))\dual})$. We have now shown that the assertion of the lemma passes to successor ordinals.

As the assertion of the lemma passes readily to limit ordinals, the proof is complete.
\end{proof}

We are now ready to determine upper estimates for the Szlenk indices of the Banach spaces $C_0([0,\,\omega^{\omega^\gamma}])$, where $\gamma$ is an arbitrary ordinal.

\begin{proposition}\label{build}
Let $\gamma$ be an ordinal and $0<n<\omega$. Then \[Sz(c_0(\omega^{\omega^\gamma},\,C_0([0,\,\omega^{\omega^\gamma n}])))\leq\omega^{\gamma+1}\, .\]
\end{proposition}

\begin{proof}
We proceed by induction on $\gamma$, first establishing the proposition in the case $\gamma=0$ and $n=1$. By Proposition~\ref{upest}, it suffices to show that
\[
\forall \eps>0 \quad s_\eps (\cball{c_0(\omega,\,C_0([0,\,\omega]))\dual})\subseteq \Big(1-\frac{\eps}{3}\Big)\cball{c_0(\omega,\,C_0([0,\,\omega]))\dual}\, .
\]
Suppose by way of contraposition that there is $\eps>0$ and $(z_l)_{l<\omega}\in s_\eps (\cball{c_0(\omega,\,C_0([0,\,\omega]))\dual})$ such that $\Vert (z_l)_{l<\omega}\Vert >1-\eps/3$. Since
\[
\Vert (z_l)_{l<\omega}\Vert = \sup \Big\{ \sum_{r\in R}\Vert J_{m,\,\omega}\dual z_r\Vert \,\, \Big\vert \,\, 0<m<\omega, \, R\subseteq \omega, \, 0<\vert R\vert <\infty \Big\}\, ,
\]
there exists $m<\omega$ and a nonempty finite set $R\subseteq \omega$ such that
\[
\sum_{r\in R}\Vert J_{m,\,\omega}\dual z_r\Vert >1-\frac{\eps}{3}\, .
\]
By Lemma~\ref{HaLag}, this implies that $(J_{m,\,\omega}\dual z_r)_{r\in R}\in s_{\eps/3}(\cball{c_0(R, \, C_0([0,\,m]))\dual})$, hence $Sz(c_0(R, \, C_0([0,\,m])))>1$. By Proposition~\ref{collection}(iii), this in turn implies that $c_0(R, \, C_0([0,\,m]))$ is infinite dimensional; however, this is impossible since $\dim (c_0(R, \, C_0([0,\,m])))=m\vert R\vert<\infty$. With this contradiction we have now established the assertion of the proposition for $\gamma=0$ and $n=1$.

Next we show that if $\beta$ is an ordinal such that the assertion of the proposition holds for $\gamma=\beta$ and $n=1$, then the proposition is true for $\gamma=\beta$ and all $0<n<\omega$. Let $1<m<\omega$ and note that, by Lemma~\ref{pb}, for any ordinal $\beta$ we have
\[
c_0(\omega^{\omega^\beta}, \, C_0([0,\,\omega^{\omega^\beta m}]))\approx c_0(\omega^{\omega^\beta}, c_0(\omega^{\omega^\beta}, \, C_0([0,\,\omega^{\omega^\beta}])))\approx c_0(\omega^{\omega^\beta}, \, C_0([0,\,\omega^{\omega^\beta}]))\, .
\]
Assuming the proposition is true for $n=1$ and $\gamma=\beta$, we deduce that
\[
Sz(c_0(\omega^{\omega^\beta}, \, C_0([0,\,\omega^{\omega^\beta m}]))) = Sz(c_0(\omega^{\omega^\beta}, \, C_0([0,\,\omega^{\omega^\beta}]))) \leq \omega^{\beta+1}\, ,
\]
as desired.

It now remains to show that if $\beta$ is an ordinal such that the assertion of the proposition holds for all $\gamma<\beta$ and $0<n<\omega$, then the assertion of the proposition holds for $n=1$ and $\gamma=\beta$. Take such $\beta$ and note that, by Proposition~\ref{upest}, it suffices to show that
\begin{equation}\label{eq2}
\forall\eps>0 \quad s_\eps^{\omega^\beta}(\cball{c_0(\omega^{\omega^\beta}, \, C_0([0,\,\omega^{\omega^\beta}]))\dual}) \subseteq \Big(1-\frac{\eps}{3}\Big)\cball{c_0(\omega^{\omega^\beta},\, C_0([0,\,\omega^{\omega^\beta}]))\dual}\, .
\end{equation}
Suppose by way of contraposition that there is $\eps>0$ and $(z_\eta)_{\eta< \omega^{\omega^\beta}}\in s_\eps^{\omega^\beta}(\cball{c_0(\omega^{\omega^\beta}, \, C_0([0,\,\omega^{\omega^\beta}]))\dual})$ with $\Vert (z_\eta)_{\eta< \omega^{\omega^\beta}}\Vert >1-\eps/3$. Since
\[
\Vert (z_\eta)_{\eta< \omega^{\omega^\beta}}\Vert = \sup \Big\{ \sum_{r\in R}\Vert J_{\omega^{\omega^\zeta m}, \, \omega^{\omega^\beta}}\dual z_r\Vert \,\, \Big\vert \,\, \zeta <\beta , \, 0<m<\omega, \, R\subseteq \omega^{\omega^\beta}, \, 0<\vert R\vert <\infty\Big\}\, ,
\]
there exists $\zeta<\beta$, $0<m<\omega$ and a nonempty finite set $R\subseteq \omega^{\omega^\beta}$ such that
\[
\sum_{r\in R}\Vert J_{\omega^{\omega^\zeta m},\, \omega^{\omega^\beta}}\dual z_r\Vert >1-\eps/3\, .
\]
By Lemma~\ref{HaLag}, this implies that $(J_{\omega^{\omega^\zeta m},\, \omega^{\omega^\beta}}\dual z_r)_{r\in R}\in s_{\eps/3}^{\omega^\beta}(\cball{c_0(R, \, C_0([0,\,\omega^{\omega^\zeta m}]))\dual})$, hence $Sz(c_0(R, \, C_0([0,\,\omega^{\omega^\zeta m}]))) > \omega^\beta$. By the induction hypothesis, it follows that
\[
\omega^{\beta}< Sz(c_0(R, \, C_0([0,\,\omega^{\omega^\zeta m}])))\leq \omega^{\zeta+1}\leq \omega^\beta\, ,
\]
which is absurd. Thus (\ref{eq2}) holds, and the assertion of the proposition holds for $n=1$ and $\gamma=\beta$. The inductive proof is now complete.
\end{proof}

\begin{theorem}\label{stpip}
Let $\alpha\geq \omega$ and let $\gamma$ be the unique ordinal satisfying $\omega^{\omega^\gamma}\leq \alpha < \omega^{\omega^{\gamma+1}}$. Then \[ Sz(C([0,\,\alpha]))=\omega^{\gamma+1}\,.\]
\end{theorem}
\begin{proof}
Let $n<\omega$ be such that $\omega^{\omega^\gamma n}>\alpha$, so that $C([0,\,\alpha])$ is isomorphic to a subspace of $C_0([0,\,\omega^{\omega^\gamma n}])$, hence isomorphic to a subspace of $c_0(\omega^{\omega^\gamma},\,C_0([0,\,\omega^{\omega^\gamma n}]))$. Then, by Proposition~\ref{collection}(i) and Proposition~\ref{build},
\begin{align}\label{prooof}
Sz(C([0,\,\alpha]))\leq Sz(c_0(\omega^{\omega^\gamma},\,C_0([0,\,\omega^{\omega^\gamma n}])))\leq \omega^{\gamma+1}.
\end{align}
To obtain the reverse inequality, we consider the functionals $\delta_\xi\in \cball{C([0,\,\alpha])\dual}$, $\xi\leq \alpha$, where $\langle \delta_\xi, \, f\rangle = f(\xi)$ for each $f\in C([0,\,\alpha])$. As the map $[0,\,\alpha]\longrightarrow C([0,\,\alpha])\dual$ is an order-$w\dual$ homeomorphic embedding, a straightforward induction shows that $\delta_{\omega^\zeta}\in s_1^{\zeta}(\cball{C([0,\,\alpha])\dual})$ whenever $\omega^\zeta \leq \alpha$. In particular, $s_1^{\omega^\gamma}(\cball{C([0,\,\alpha])\dual})\ni \delta_{\omega^{\omega^\gamma}}$ is nonempty, hence $Sz(C([0,\,\alpha]))>\omega^\gamma$. By Proposition~\ref{collection}(ii), $Sz(C([0,\,\alpha]))\geq\omega^{\gamma+1}$, and we are done.
\end{proof}

\section{The $w^\ast$-dentability index of $C([0,\,\alpha])$}\label{secthree}

In this section we discuss the $w^\ast$-dentability indices of the spaces $C([0,\,\alpha])$, where $\alpha$ is an arbitrary ordinal. For a (real) Asplund space $X$, the definitions of the $\eps$-$w^\ast$-dentability index $Dz(X,\,\eps)$ of $X$ and the $w^\ast$-dentability index $Dz(X)$ of $X$ are essentially the same as for the Szlenk indices $Sz(X, \, \eps)$ and $Sz(X)$, the difference being that in the derivation on $w^\ast$-compact sets we remove only $w^\ast$-slices of small norm diameter (for $x\in X$ and $t\in \mathbb{R}$, let $H(x,\,t)= \{x\dual \in X\dual \mid x\dual (x)>t\}$; for $B\subseteq X\dual$, a $w\dual$-\emph{slice of} $B$ is any set of the form $H(x,\,t)\cap B$, where $x\in X$ and $t\in \mathbb{R}$.) To be precise, let $X$ be an Asplund space and $B \subseteq X\dual$. Define
\begin{equation*}
d_\eps(B) = \{x\dual \in B \mid \diam (B \cap H(x,\,t))> \eps \mbox{ for every } w\dual \mbox{-slice }H(x,\,t)\ni x\dual\}\,.
\end{equation*} We iterate $d_\eps$ transfinitely, setting $d_\eps^0(B) = B$, $ d_\eps^{\beta+1}(B)= d_\eps(d_\eps^\beta(B))$ for each ordinal $\beta$ and $ d_\eps^\beta(B) = \bigcap_{\sigma< \beta} d_\eps^\sigma(B)$ whenever $\beta $ is a limit ordinal. 

Define ${Dz}(X,\,\eps)$ to be the first ordinal $\beta$ such that $d_\eps^\beta(\cball{X\dual}) = \emptyset$, and $Dz(X) :=\sup_{\eps >0}{Dz}(X,\,\eps)$. Similarly to the Szlenk index, the $w^\ast$-dentability index ${Dz}(X)$ is defined for every Asplund space $X$.

The natural analogues of parts (i) and (ii) of Proposition~\ref{collection} hold for the $w^\ast$-dentability index, with similar proofs. In particular, $Dz(X)\leq Dz(Y)$ whenever $X$ is a subspace of $Y$, and $Dz(X)>\omega^\gamma$ implies $Dz(X)\geq \omega^{\gamma+1}$. For part (iii), the analogous result for the $w^\ast$-dentability index is that $Dz(X)\leq \omega$ if and only if $X$ is superreflexive; this is due to Lancien \cite{Lancien1995}. Moreover, it is clear that ${Sz}(X)\leq {Dz}(X)$ for all Asplund spaces $X$; conversely, we have the following:

\begin{proposition}[\cite{Lancien2006}]\label{bochnerest}
Let $X$ be an Asplund space and $L_2(X)$ the Banach space of all (equivalence classes of) Bochner integrable functions $f:[0,\,1]\longrightarrow X$, equipped with its usual norm. Then \[  Dz(X)\leq Sz(L_2(X))\,. \]
\end{proposition}

Proposition~\ref{bochnerest} was used in \cite{H'ajek2009} to show that for ordinals $\omega^{\omega^\gamma}\leq \alpha<\omega^{\omega^{\gamma+1}}<\omega_1$, the $w^\ast$-dentability index of $C([0,\,\alpha])$ is $\omega^{1+\gamma+1}$. The authors of \cite{H'ajek2009} then extend their result to a certain nonseparable setting by showing that for a scattered compact Hausdorff space $K$ of countable height, the $w^\ast$-dentability index of $C(K)$ is $\omega^{1+\gamma+1}$, where $\gamma$ is the unique (countable) ordinal such that the height of $K$ belongs to the ordinal interval $[\omega^\gamma, \, \omega^{\gamma+1})$. The following result extends the main result of \cite{H'ajek2009} in a different direction.

\begin{theorem}\label{dentthm}
Let $\alpha\geq \omega$ and let $\gamma$ be the unique ordinal satisfying $\omega^{\omega^\gamma}\leq \alpha < \omega^{\omega^{\gamma+1}}$. Then \[ Dz(C([0,\,\alpha]))=\omega^{1+\gamma+1}\,.\]
\end{theorem}

We shall only sketch the proof of Theorem~\ref{dentthm}, as the differences between the proofs of Theorem~\ref{stpip} and Theorem~\ref{dentthm} are completely analogous to the differences between the proofs of the separable cases established in \cite{H'ajek2007} and \cite{H'ajek2009} (we note that although it is essentially possible to prove Theorem~\ref{stpip} and Theorem~\ref{dentthm} simultaneously by estimating the Szlenk index of $L_2(\mu, \, C([0,\,\alpha]))$, where $\mu$ is assumed to be either counting measure on a singleton or Lebesgue measure on $[0,\,1]$, respectively, we feel it would obscure the main ideas of the proof of Theorem~\ref{stpip} to do so). Theorem~\ref{dentthm} follows readily from the Szlenk index estimate given by the following result.

\begin{proposition}\label{dentdentest}
Let $\gamma$ be an ordinal and $0<n<\omega$. Then
\[
Sz(L_2(c_0(\omega^{\omega^\gamma},\,C([0,\,\omega^{\omega^\gamma n}])))) \leq \omega^{1+\gamma+1}\,.
\]
\end{proposition}

The main difficulty in establishing Proposition~\ref{dentdentest} is to prove the following variant of Proposition~\ref{HaLag}; the proof combines ideas from the proofs of Proposition~\ref{dentdentest} and \cite[Lemma~6]{H'ajek2009}

\begin{lemma}
Let $0<n<\omega$ and let $\zeta$ and $\gamma$ be ordinals satisfying either $\zeta=\gamma=0$ or $\omega^{\omega^\zeta n}<\omega^{\omega^\gamma}$. Let $\emptyset \subsetneq R\subseteq \omega^{\omega^\gamma}$, $\eps >0$ and let $J$ denote the canonical embedding of $L_2(c_0(R,\,C([0,\,\omega^{\omega^\zeta n}])))$ into $L_2(c_0(\omega^{\omega^\gamma},\,C([0,\,\omega^{\omega^\gamma}])))$. Let $\beta$ be an ordinal. If $z\in s_{3\eps}^\beta(B_{L_2(c_0(\omega^{\omega^\gamma},\,C([0,\,\omega^{\omega^\gamma}])))^\ast})$ and $\Vert J^\ast z\Vert^2 >1-\eps^2 $, then $J^\ast z\in s_\eps^\beta (B_{L_2(c_0(R,\,C([0,\,\omega^{\omega^\zeta n}])))^\ast})$.
\end{lemma}

The estimate $Dz(C([0,\,\alpha]))\leq \omega^{1+\gamma+1}$ follows readily from Proposition~\ref{bochnerest} and Proposition~\ref{dentdentest}. For the reverse inequality, note that for the case $\gamma\geq \omega$ we have \[ Dz(C([0,\,\alpha]))\geq Sz(C([0,\,\alpha])) > Sz(C([0,\,\alpha]), \, 1)\geq  \omega^{\gamma} =\omega^{1+\gamma}\,,\] so that the required estimate follows by the aforementioned $w^\ast$-dentability index version of Proposition~\ref{collection}(ii). The case $\gamma<\omega$ follows from the fact established in \cite[Proposition~11]{H'ajek2009} that $Dz(C([0,\,\omega^{\omega^\gamma}]),\,1/2)>\omega^{1+\gamma}$ for every $\gamma<\omega$.

\end{document}